\documentclass[11pt,a4paper]{amsart}

\usepackage{amsmath,amssymb,amsfonts,amsthm,enumerate,tikz,comment} 
\usepackage[british]{babel}

\title[Negligibility of parabolic elements]{Negligibility of parabolic elements in relatively hyperbolic groups}
\author{Motiejus Valiunas}
\date{\today}
\address{Department of Mathematics, University of Southampton, University Road, Southampton SO17~1BJ}
\email{m.valiunas@soton.ac.uk}

\newcommand{\PP}{\mathcal{P}}
\newcommand{\HH}{\mathcal{H}}
\newcommand{\Z}{\mathbb{Z}}
\newcommand{\Arel}{\operatorname{Area}_{rel}}
\newcommand{\Cay}{\Gamma}
\newcommand{\dc}{\operatorname{dc}}

\theoremstyle{definition}
\newtheorem{defn}{Definition}[section]
\newtheorem{rmk}[defn]{Remark}
\newtheorem*{ntn}{Notation}
\newtheorem*{ack}{Acknowledgements}
\theoremstyle{plain}
\newtheorem{thm}[defn]{Theorem}
\newtheorem{cor}[defn]{Corollary}
\newtheorem{prop}[defn]{Proposition}
\newtheorem{lem}[defn]{Lemma}
\theoremstyle{remark}

\usetikzlibrary{decorations.markings}
\tikzset{->-/.style={decoration={
  markings,
  mark=at position .5 with {\arrow{>}}},postaction={decorate}}}

\begin{document}

\begin{abstract}
We study density of parabolic elements in a finitely generated relatively hyperbolic group $G$ with respect to a word metric. We prove this density to be zero (apart from degenerate cases) and the limit defining the density to converge exponentially fast; this has recently been proven independently by W.\ Yang in \cite{yanggen}. As a corollary, we obtain the analogous result for the set of commuting pairs of elements in $G^2$, showing that the degree of commutativity of $G$ is equal to zero.
\end{abstract}
\maketitle

\section{Introduction}
\label{s:intro}

A group $G$ is hyperbolic to a collection of subgroups $\{ H_\omega \}_{\omega \in \Omega}$ if, loosely speaking, it is hyperbolic except for the part that is inside the set $\PP$ consisting of elements in the subgroups $H_\omega$ and their conjugates. It is therefore natural to ask whether taking elements from $G$ ``at random'' we can expect these elements to be outside $\PP$ and therefore ``behave like in a hyperbolic group''. We prove that this is the case if $G$ is finitely generated and the sequence of measures that makes sense of the words ``at random'' comes from a word metric on $G$.

More precisely, let $G$ be a finitely generated group and let $X$ be a finite generating set for $G$. Denote by $|\cdot|_X: G \to \Z_{\geq 0}$ the \emph{word metric} on $G$ with respect to $X$. For any $n \in \Z_{\geq 0}$, define the sets $$S_{G,X}(n) := \{ g \in G \mid |g|_X = n \},$$ the \emph{sphere} of radius $n$ in the Cayley graph $\Cay(G,X)$, and $$B_{G,X}(n) := \{ g \in G \mid |g|_X \leq n \} = \bigcup_{i=0}^n S_{G,X}(n),$$ the \emph{ball} of radius $n$ in $\Cay(G,X)$. The following definition can be used to characterise ``small'' subsets of $G$. The term ``negligible'' to describe small subsets of $G^r$ (for a finitely generated infinite group $G$) was coined in \cite{kapovich}, although the definition given therein is not equivalent to Definition \ref{d:negl} here; in the case $r = 1$, the following definition is used implicitly in \cite{burillo}.

\begin{defn} \label{d:negl}
Let $r \geq 1$, and let $\mathcal{S} \subseteq G^r$ be a subset. For $n \geq 0$, let $$\delta_X(\mathcal{S},n) := \frac{|\mathcal{S} \cap B_{G,X}(n)^r|}{|B_{G,X}(n)|^r}$$ be the fraction of elements in $B_{G,X}(n)^r$ that belong to $\mathcal{S}$. The set $\mathcal{S}$ is said to be \emph{negligible} in $G$ with respect to $X$ if $\delta_X(\mathcal{S},n) \to 0$ as $n \to \infty$. Moreover, $\mathcal{S}$ is said to be \emph{exponentially negligible} in $G$ with respect to $X$ if in addition there exists a constant $\rho > 1$ such that $\delta_X(\mathcal{S},n) \leq \rho^{-n}$ for all sufficiently large $n$.
\end{defn}

There are various definitions of relatively hyperbolic groups, due to M.~Gromov \cite{gromov}, B.~Farb \cite{farb}, C.~Dru\c{t}u \& M.~Sapir \cite{drutu}, D.~V.~Osin \cite{osin06def}, D.~Groves \& J.~S.~Manning \cite{groves}, and B.~H.~Bowditch \cite{bowditch}. In this paper we use the definition by Osin; for a precise statement, see Section \ref{s:prelim}. Our main result is as follows:

\begin{thm}
\label{t:main}
Let $G$ be a finitely generated group that is not virtually cyclic, and let $X$ be a finite generating set. Suppose that $G$ is hyperbolic with respect to a collection of \emph{proper} subgroups $\{ H_\omega \}_{\omega \in \Omega}$. Let $$\mathcal{P} := \bigcup_{\substack{\omega \in \Omega \\ g \in G}} H_\omega^g$$ be the set of \emph{parabolic} elements of $G$. Then $\mathcal{P}$ is exponentially negligible in $G$ with respect to $X$.
\end{thm}

\begin{rmk}
During the process of writing up this paper, the author has discovered a more general result by W.~Yang in \cite{yanggen}. In particular, Theorem 1.7 therein is the same as Theorem \ref{t:main} above, and it is closely related to a genericity result that works in a more general setting \cite[Theorem A]{yanggen}. Thus most of this paper merely gives an alternative proof to a recent but already-known result.
\end{rmk}

As an immediate consequence of the Theorem we obtain:
\begin{cor}
\label{c:loxo}
Let $G$ and $X$ be as in Theorem \ref{t:main}. Let $\mathcal{Q} \subseteq G$ be the set of finite order elements. Then $\mathcal{P} \cup \mathcal{Q}$ is exponentially negligible in $G$ with respect to $X$.
\end{cor}

The next result computes the degree of commutativity of relatively hyperbolic groups. The \emph{degree of commutativity} of a finitely generated group $G$ with respect to a finite generating set $X$ was defined by \cite{amv} as
$$\dc_X(G) := \limsup_{n \to \infty} \frac{|\{ (x,y) \in B_{G,X}(n)^2 \mid xy=yx \}|}{|B_{G,X}(n)|^2}.$$
It has been conjectured \cite[Conjecture 1.6]{amv} that $\dc_X(G) = 0$ whenever $G$ is not virtually abelian (independently of $X$). The next corollary confirms the conclusion of the conjecture in the case when $G$ is a non-elementary relatively hyperbolic group, thereby generalising the result for hyperbolic groups in \cite[Theorem 1.7]{amv}.

\begin{cor}
\label{c:dc}
Let $G$ and $X$ be as in Theorem \ref{t:main}. Then the set of pairs of commuting elements, $\{ (x,y) \in G^2 \mid xy=yx \}$, is exponentially negligible in $G$ with respect to $X$.
\end{cor}

The structure of the paper is as follows. Section \ref{s:prelim} defines relatively hyperbolic groups and recalls some of the main results on the geometry of their Cayley graphs. Section \ref{s:geod} derives some further results relating geodesics and quasi-geodesics of the ``usual'' (i.e.\ locally compact) and ``coned-off'' (cf terminology in \cite{farb}) Cayley graphs. We give a proof of Theorem \ref{t:main} in Section~\ref{s:pftmain}, and proofs of Corollaries \ref{c:loxo} and \ref{c:dc} in Section \ref{s:cor}.

\begin{ntn}
For a group $G$ and a generating subset $Z \subseteq G$, we will write $Z^\ast$ for the set of all words over $Z \cup Z^{-1}$, and we will identify these words in the obvious way with paths starting at $1 \in G$ in the Cayley graph $\Cay(G,Z)$ (we do not require $Z$ to be finite and so $\Cay(G,Z)$ to be locally finite). Moreover, we will identify $G$ with the vertices of $\Cay(G,Z)$. Given a path $P$ in $\Cay(G,Z)$, we also say it is \emph{labelled} by a word $Q \in Z^\ast$ if $P = gQ$ (viewed as paths) for some $g \in G$, and we let $P_-$ (resp.\ $P_+$) be the starting (resp.\ ending) vertex of $P$. For a word $P \in Z^\ast$, we will write $\ell_Z(P)$ for the length of the path $P$ in $\Cay(G,Z)$ (= number of letters in $P \in Z^\ast$), and we will write $\overline{P} = \overline{P}_G$ for the corresponding element of $G$. For an element $g \in G$, we will write $|g|_Z$ for the word length of $g$ with respect to $Z$; in particular, if $p$, $q$ are vertices in $\Cay(G,Z)$, then the distance between them will be $|p^{-1}q|_Z$.
\end{ntn}

\begin{ack}
The author would like to thank his PhD advisor Armando Martino, without whose guidance this paper would not have been possible. The author is also grateful to Yago Antol\'in for valuable discussions and advice.
\end{ack}

\section{Preliminaries}
\label{s:prelim}

We use Osin's definition of relative hyperbolicity, given in \cite{osin06def}. For this, let $G$ be a group, $\{ H_\omega \}_{\omega \in \Omega}$ a collection of subgroups of $G$, and $X \subseteq G$ a subset. Define the group $$F := \left(\ast_{\omega \in \Omega} H_\omega\right) \ast F(X),$$ and let $\varphi: F \to G$ be the canonical homomorphism. If there exists a \emph{finite} subset $X \subseteq G$ as above such that $\varphi$ is surjective and $\ker(\varphi)$ is the normal closure of a \emph{finite} set $\mathcal{R} \subseteq F$, then $G$ is said to be \emph{finitely presented relative to} $\{ H_\omega \}_{\omega \in \Omega}$.

Moreover, if we define $$\HH := \bigcup_{\omega \in \Omega} \left( H_\omega \setminus \{1\} \right),$$ then any word $P \in \left( X \cup \HH \right)^\ast$ such that the image $\overline{P} = \overline{P}_F$ of $P$ in $F$ is in the kernel of $\varphi$ satisfies an equality $$\overline{P} = \prod_{i=1}^n f_i^{-1} r_i^{\varepsilon_i} f_i$$ for some $f_i \in F$, $r_i \in \mathcal{R}$ and $\varepsilon_i \in \{ \pm 1 \}$; let $\Arel(P)$ be the minimal value of $n$ such that $\overline{P}$ can be written as above. If there exists a constant $C \geq 0$ such that $$\Arel(P) \leq C \ell_{X \cup \HH}(P)$$ for every $P \in \left( X \cup \HH \right)^\ast$ such that $\overline{P}_G = 0$, then $G$ is said to \emph{satisfy a relative linear isoperimetric inequality} (with respect to $X$ and $\{ H_\omega \}_{\omega \in \Omega}$).

\begin{defn}
\label{d:rh}
The group $G$ is said to be \emph{hyperbolic relative to} $\{ H_\omega \}_{\omega \in \Omega}$ if it is finitely presented with respect to $\{ H_\omega \}_{\omega \in \Omega}$ and satisfies a relative linear isoperimetric inequality. We call the $H_\omega$ the \emph{peripheral subgroups} of $G$, and say $G$ is \emph{relatively hyperbolic} if it is hyperbolic relative to some collection of peripheral subgroups.
\end{defn}

For the remainder of the paper we fix a group $G$ and a collection of \emph{proper} subgroups $\{ H_\omega \}_{\omega \in \Omega}$, such that $G$ is hyperbolic relative to $\{ H_\omega \}_{\omega \in \Omega}$ (given some finite subset $X \subseteq G$, which is also fixed for now). We will usually assume that, moreover, $G$ is finitely generated and $X$ is a (finite) generating set.

It is worth noting that in this case Definition \ref{d:rh} is independent of a chosen generating set $X$. Indeed, given two finite generating sets $X$ and $Y$, suppose $G$ is finitely presented relative to $X$. Then $G$ is also finitely presented relative to $Y$ since the canonical homomorphism $$\widetilde\varphi: \widetilde{F} := \left(\ast_{\omega \in \Omega} H_\omega\right) \ast F(Y) \to G$$ is surjective and $$\ker(\widetilde\varphi) = \langle\!\langle \{ \psi_Y(r) \mid r \in \mathcal{R} \} \cup \{ y^{-1}\psi_Y(\psi_X(y)) \mid y \in Y \} \rangle\!\rangle^{\widetilde{F}},$$ where $\mathcal{R}$ is as above, and $\psi_X(P)$ (resp.\ $\psi_Y(P)$) is a word over $X \cup \HH$ (resp.\ $Y \cup \HH$) obtained by replacing every letter of $Y$ (resp.\ $X$) in $P \in (Y \cup \HH)^\ast$ (resp.\ $P \in (X \cup \HH)^\ast$) by a word over $X$ (resp.\ $Y$) representing the same element in $G$. Moreover, \cite[Theorem 2.34]{osin06def} says that $G$ satisfies a linear isoperimetric inequality with respect to $X$ if and only if $G$ satisfies it with respect to $Y$.

The definition below summarises common terms used to describe paths in ${\Cay(G,X \cup \HH)}$. The endpoints of paths in Cayley graphs that we consider will always be vertices.
\begin{defn}
Let $P$ be a path in the Cayley graph $\Cay(G,X \cup \HH)$. We say that \begin{enumerate}[(i)]
\item a subpath $Q$ of $P$ is an \emph{$H_\omega$-subpath} if it is labelled by a word from $(H_\omega)^\ast$, with a convention that a single vertex is an $H_\omega$-subpath for any $\omega \in \Omega$;
\item a subpath $Q$ of $P$ is an \emph{($H_\omega$-)component} if it is a maximal $H_\omega$-subpath, and a maximal $H_\omega$-subpath $Q$ is called a \emph{trivial ($H_\omega$-)component} if it is a single vertex;
\item a vertex $p$ of $P$ is \emph{non-phase} if it is an interior vertex of some component of $P$, and $p$ is \emph{phase} otherwise;
\item two $H_\omega$-components $Q_1$, $Q_2$ of paths $P_1$, $P_2$, respectively, in the graph $\Cay(G,X \cup \HH)$ are \emph{connected} if there is an edge from $(Q_1)_-$ to $(Q_2)_-$ labelled by an element of $H_\omega$;
\item an $H_\omega$-component $Q$ of $P$ is \emph{isolated} if it is not connected to any other $H_\omega$-component of $P$;
\item the path $P$ \emph{does not vertex backtrack} (resp.\ \emph{does not backtrack}) if all its components (resp.\ all its non-trivial components) are isolated, and \emph{vertex backtracks} (resp.\ \emph{backtracks}) otherwise.
\end{enumerate}
\end{defn}
It is clear that if $P$ is a geodesic in $\Cay(G,X \cup \HH)$, then $P$ does not vertex backtrack and all its vertices are phase.

We are interested in the collection $\mathcal{P}$ of parabolic elements of $G$.
\begin{defn}
An element $g \in G$ is \emph{parabolic} if it is conjugate to some element of $H_\omega$ for some $\omega \in \Omega$, and $g$ is \emph{hyperbolic} otherwise. We denote by $$\mathcal{P} := \bigcup_{\substack{\omega \in \Omega \\ g \in G}} H_\omega^g$$ the set of parabolic elements of $G$.
\end{defn}

We now recall some of the results about relatively hyperbolic groups which will be used in this paper. The first of them is a stronger version of the statement that the graph $\Cay(G,X \cup \HH)$ is Gromov-hyperbolic.

\begin{thm}[\cite{osin06def}, Theorem 3.26] \label{t:hyp}
There exists a constant $\nu \in \Z_{\geq 1}$ with the following property. Let $\Delta \subseteq \Cay(G,X \cup \HH)$ be a geodesic triangle with edges $P$, $Q$ and $R$, and let $p \in P$ be a vertex. Then there exists a vertex $q \in Q \cup R$ such that $$|p^{-1}q|_X \leq \nu.$$
\end{thm}

The second result introduces what is known as the \emph{bounded coset penetration (BCP) property}. For this, recall the following definition.

\begin{defn}
Let $(K,d)$ be a geodesic metric space, $\lambda \geq 1$, and $c \geq 0$. A path $\alpha: [0,\ell_\alpha] \to K$ parametrised by arc length is a \emph{$(\lambda,c)$-quasi-geodesic} in $K$ if $$|t_1-t_2| \leq \lambda d(\alpha(t_1),\alpha(t_2))+c$$ for all $t_1,t_2 \in [0,\ell_\alpha]$.
\end{defn}

\begin{thm}[\cite{osin06def}, Theorem 3.23] \label{t:bcp}
For any given $\lambda \geq 1$ and $c \geq 0$, there exists a constant $\varepsilon = \varepsilon(\lambda,c) \geq 0$ with the following property. Let $P$ and $Q$ be $(\lambda,c)$-quasi-geodesic paths in $\Cay(G,X \cup \HH)$ that do not backtrack such that $P_- = Q_-$ and $P_+ = Q_+$. Then \begin{enumerate}[(i)]
\item \label{i:t-bcp-near} If $p \in P$ is a phase vertex, then there exists a phase vertex $q \in Q$ such that $|p^{-1}q|_X \leq \varepsilon$.
\item \label{i:t-bcp-conn}	If $R$ is a non-trivial $H_\omega$-component of $P$ and $|R_-^{-1}R_+|_X > \varepsilon$, then there exists a non-trivial $H_\omega$-component of $Q$ that is connected to $R$.
\item \label{i:t-bcp-ste} If $R \subseteq P$ and $S \subseteq Q$ are connected non-trivial $H_\omega$-components, then $$\max \{ |R_-^{-1}S_-|_X, |R_+^{-1}S_+|_X \} \leq \varepsilon.$$
\end{enumerate}
\end{thm}

The following (perhaps less standard) result allows us to enlarge an arbitrary finite generating set $Y$ of $G$ to a ``nicer'' set $X$ such that geodesics in $\Cay(G,X)$ can be related to quasi-geodesics in $\Cay(G,X \cup \HH)$. For this we need to construct derived paths, defined by by Antol\'in \& Ciobanu in \cite[Construction 4.1]{ac} (to avoid unnecessary complications, we will only define this for geodesics).

\begin{defn}
Let $X$ be a finite generating set for $G$, let $P$ be a geodesic path in $\Cay(G,X)$. We can (uniquelly) express $P$ as a concatenation $$P = A_0U_1A_1\cdots U_nA_n,$$ where the $U_i$ are labelled by non-trivial words in $(H_{\omega_i})^\ast$ for some $\omega_i \in \Omega$, and no $U_i$ is a proper subpath of a subpath $Q$ of $P$ such that $(U_i)_- = Q_-$ and $Q$ is labelled by a word in some $(H_\omega)^\ast$.
\begin{enumerate}[(i)]
\item The \emph{derived path} $\widehat{P}$ of $P$ is a path in $\Cay(G,X \cup \HH)$ given by $$\widehat{P} := A_0h_1A_1 \cdots h_nA_n,$$ where $h_n$ is an edge labelled by an element of $H_{\omega_i}$ such that $\overline{h_i} = \overline{U_i}$ in $G$. 
\item We call a vertex $p \in P$ a \emph{phase vertex} of $P$ if it is not an interior vertex of any of the $U_i$ (and so ``survives'' in $\widehat{P}$).
\end{enumerate}
\end{defn}

\begin{thm}[\cite{ac}, Lemma 5.3] \label{t:gsl}
Let $Y$ be an arbitrary generating set for $G$. Then there exist $\lambda \geq 1$, $c \geq 0$ and a finite subset $\HH'$ of $\HH$ such that for every finite subset $X$ of $G$ satisfying $$Y \cup \HH' \subseteq X \subseteq Y \cup \HH$$ and for any geodesic path $P$ in $\Cay(G,X)$, the derived path $\widehat{P}$ in $\Cay(G,X \cup \HH)$ is a $(\lambda,c)$-quasi-geodesic that does not vertex backtrack.
\end{thm}

A generating set $X$ satisfying the conclusion of Theorem \ref{t:gsl} will be called a \emph{well-behaved} generating set.

Finally, we have the following finiteness results.

\begin{thm}[\cite{osin06def}, Corollary 2.48] \label{t:finmanypar}
If $G$ is finitely generated, then we have $|\Omega| < \infty$.
\end{thm}

\begin{thm}[\cite{osin06def}, Proposition 2.36] \label{t:almaln}
If $H_\omega \cap H_{\widetilde\omega}^g$ is infinite for some $\omega,\widetilde\omega \in \Omega$ and $g \in G$, then $\widetilde\omega = \omega$ and $g \in H_\omega$.
\end{thm}

\section{Geodesics in Cayley graphs}
\label{s:geod}

Combining Theorems \ref{t:hyp}, \ref{t:bcp}\eqref{i:t-bcp-near} and \ref{t:gsl} it is easy to see that we have the following result. In particular, we may take $\tilde\delta := \nu + 2\varepsilon(\lambda,c)$, where $\nu$ and $(\lambda,c)$ are given by Theorems \ref{t:hyp} and \ref{t:gsl}, respectively, and $\varepsilon(\lambda,c)$ is given by Theorem \ref{t:bcp}.

\begin{cor} \label{c:hyp}
Let $X$ be a well-behaved generating set of $G$. There exists a constant $\tilde\delta \geq 0$ with the following property. Consider a geodesic triangle in $\Cay(G,X)$ formed by edges $P$, $Q$ and $R$, and let $p$ be a phase vertex of $P$. Then there exists a phase vertex $q$ of either $Q$ or $R$ such that $$|p^{-1}q|_X \leq \tilde\delta. \eqno\qed$$
\end{cor}

This implies that for a path $P$ in $\Cay(G,X)$ that is ``not too long'', phase vertices of the geodesic with endpoints $P_-$, $P_+$ are ``not too far'' from $P$.

\begin{prop} \label{p:logclose}
Let $X$ be a well-behaved generating set of $G$, and let $P$ be a path in $\Cay(G,X)$. Let $Q$ be a geodesic in $\Cay(G,X)$ with endpoints $Q_- =P_-$, $Q_+=P_+$, and let $q$ be a phase vertex of $Q$. Then there exists a vertex $p$ of $P$ such that $$|p^{-1}q|_X \leq \tilde\delta \left\lceil \log_2(\ell_X(P)) \right\rceil$$ for a universal constant $\tilde\delta \geq 0$.
\end{prop}

\begin{proof}
Let $\tilde\delta \geq 0$ be the constant given by Corollary \ref{c:hyp}, and let $$s := \left\lceil \log_2(\ell_X(P)) \right\rceil.$$ The proof resembles one that proves that geodesics in a hyperbolic metric space diverge exponentially, see e.g.\ \cite[Lemma 7.1.A]{gromov}.

We start with the geodesic $Q$ and, for $b$ a binary string of length $\leq s$, define the geodesics $Q_b$ as follows. Suppose that the geodesic $Q_b$ has been defined for a binary string $b$ of length $< s$. Let $m_b$ be a vertex on $P$ such that $$\left| \ell_X(P_{b0})-\ell_X(P_{b1}) \right| \leq 1$$ where $P_{b0}$ (resp.\ $P_{b1}$) is a subpath of $P$ with endpoints $(Q_b)_-$ and $m_b$ (resp.\ $m_b$ and $(Q_b)_+$). Then we define $Q_{b0}$ (resp.\ $Q_{b1}$) to be a geodesic with endpoints $(Q_b)_-$ and $m_b$ (resp.\ $m_b$ and $(Q_b)_+$). Note that if $b$ has length $s$, then $\ell_X(Q_b) \leq 1$ and so $Q_b$ is a subpath of $P$.

Now let $q = q_0$ be a phase vertex of $Q$, and construct phase vertices $q_i$ of $Q_{b(i)}$, for $1 \leq i \leq s$ and $b(i)$ a binary string of length $i$, as follows. Suppose $q_j$ and $b(j)$ have been chosen for $0 \leq j \leq i$, for some $i < s$. Consider the geodesic triangle formed by edges $Q_{b(i)}$, $Q_{b(i)0}$ and $Q_{b(i)1}$. Then by Corollary \ref{c:hyp}, for some $c \in \{ 0,1 \}$ there exists a phase vertex $q_{i+1}$ of $Q_{b(i+1)}$, where $b(i+1) = b(i)c$, such that $|q_i^{-1} q_{i+1}|_X \leq \tilde\delta$.

Finally, $Q_{b(s)}$ is a subpath of $P$, so in particular $q_s \in P$, and we get $$|q^{-1}q_s|_X \leq \sum_{i=0}^{s-1} |q_i^{-1}q_{i+1}|_X \leq \tilde\delta s,$$ as required.
\end{proof}

In particular, as a corollary we obtain the following result.

\begin{cor} \label{c:logclose}
Let $Y$ be any finite generating set for $G$. Then there exist constants $\tilde\lambda \geq 0$ and $\tilde{c} \geq 0$ such that the following holds. Let $P$ be a geodesic path in $\Cay(G,Y)$, and let $Q$ be a geodesic path in $\Cay(G,Y \cup \HH)$ such that $Q_- = \widehat{P}_-$ and $Q_+ = \widehat{P}_+$. Then for any vertex $q$ of $Q$, there exists a vertex $p$ of $P$ such that $$|p^{-1}q|_Y \leq \tilde\lambda \log_2(\ell_Y(P)) + \tilde{c}.$$
\end{cor}

\begin{rmk}
In fact, the conclusion of Corollary \ref{c:logclose} can be strengthened by further requiring a constant bound on $|p^{-1}q|_Y$ that is independent of $P$, i.e.\ we can further assume that $\tilde\lambda = 0$, and (moreover) it is enough to require $P$ to be a quasi-geodesic. This is shown in \cite[Lemma 8.8]{hruska}. However, for the purposes of proving Theorem \ref{t:main}, the conclusion of Corollary \ref{c:logclose} is enough.
\end{rmk}

\begin{proof}[Proof of Corollary \ref{c:logclose}]
Let $X$ be the well-behaved finite generating set containing $Y$ given by Theorem \ref{t:gsl}, and note that $X \cup \HH = Y \cup \HH$. Let $\lambda_X$ be the constant of the bilipschitz equivalence of $Y$ and $X$, i.e.\ a constant such that $|g|_Y \leq \lambda_X |g|_X$ for any $g \in G$ (we may take $\lambda_X := \max \{ |x|_Y \mid x \in X \}$). Given the generating set $X$, let $(\lambda,c)$ and $\tilde\delta$ be given by Theorem \ref{t:gsl} and Corollary \ref{c:hyp}, respectively, and let $\varepsilon(\lambda,c)$ be given by Theorem \ref{t:bcp}.

Now let $R$ be a geodesic path in $\Cay(G,X)$ with $R_- = i(P)_-$ and $R_+ = i(P)_+$, where $i: \Cay(G,Y) \hookrightarrow \Cay(G,X)$ is the canonical inclusion. Let $q \in Q$ be a vertex; note that, since $Q$ is a geodesic, $q$ is necessarily a phase vertex. By Theorem \ref{t:gsl}, $\widehat{R}$ is a $(\lambda,c)$-quasi-geodesic, and so by Theorem \ref{t:bcp}\eqref{i:t-bcp-near}, there exists a vertex $r$ of $\widehat{R}$ (viewed also as a phase vertex $r$ of $R$) such that $|r^{-1}q|_X \leq \varepsilon(\lambda,c)$. Now let $p \in P$ (technically, $i(p) \in i(P)$) be the vertex given by applying Proposition \ref{p:logclose} to the path $i(P)$, the geodesic $R$ and the phase vertex $r$ of $R$. Thus we have
\begin{align*}
\frac{1}{\lambda_X} |p^{-1}q|_Y &\leq |p^{-1}q|_X \leq |p^{-1}r|_X + |r^{-1}q|_X \leq \tilde\delta \left\lceil \log_2(\ell_X(P)) \right\rceil + \varepsilon(\lambda,c) \\ &\leq \tilde\delta \left( \log_2(\ell_Y(P))+1 \right) + \varepsilon(\lambda,c),
\end{align*}
so setting $\tilde\lambda := \lambda_Y \tilde\delta$ and $\tilde{c} := \lambda_Y \left( \tilde\delta + \varepsilon(\lambda,c) \right)$ gives the result.
\end{proof}

In particular, note that with $P$, $p$ and $q$ as in Corollary \ref{c:logclose}, the number $|p^{-1}q|_Y$ has a sublinear upper bound in terms of $\ell_Y(P)$. We will fix the constants $\tilde\lambda$ and $\tilde{c}$ given by Corollary \ref{c:logclose} for the remainder of this section, which gives a proof of the following Theorem.

\begin{thm} \label{t:parbound1}
Suppose $G$ is not virtually cyclic, and let $X$ be a finite generating set for $G$. Then we have $$|\mathcal{P} \cap B_{G,X}(n)| \leq D \sum_{\omega \in \Omega} \sum_{i=0}^{\left\lfloor \frac{n+f(n)}{2} \right\rfloor} |S_{G,X}(i)| |H_\omega \cap B_{G,X}(n+f(n)-2i)|$$ for some $D \geq 0$ and some function $f: \Z_{\geq 0} \to \Z_{\geq 0}$ such that $\frac{f(n)}{n} \to 0$ as $n \to \infty$.
\end{thm}

Let $\hat{h} \in \mathcal{P}$ be an arbitrary parabolic element; by increasing the constant $D$ if necessary we may assume that $\hat{h} \neq 1$. Thus $\hat{h} = ghg^{-1}$ for some $g \in G$ and $h \in H_\omega$ for some $\omega \in \Omega$; choose $(g,h)$ in such a way that $|g|_{X \cup \HH}$ is minimal. Consider the conjugacy diagram $R_1 Q R_2^{-1} P^{-1}$, where the paths $P$, $Q$, $R_1$ and $R_2$ are geodesics in $\Cay(G, X \cup \HH)$ such that $\overline{P}_G = \hat{h}$, $\overline{Q}_G = h$ and $\overline{(R_1)}_G = \overline{(R_2)}_G = g$; note that $Q$ is a single edge. Let $P_0$ be a geodesic in $\Cay(G,X)$ such that $(\widehat{P_0})_- = P_-$ and $(\widehat{P_0})_+ = P_+$. See Figure \ref{f:mapto3}, left.

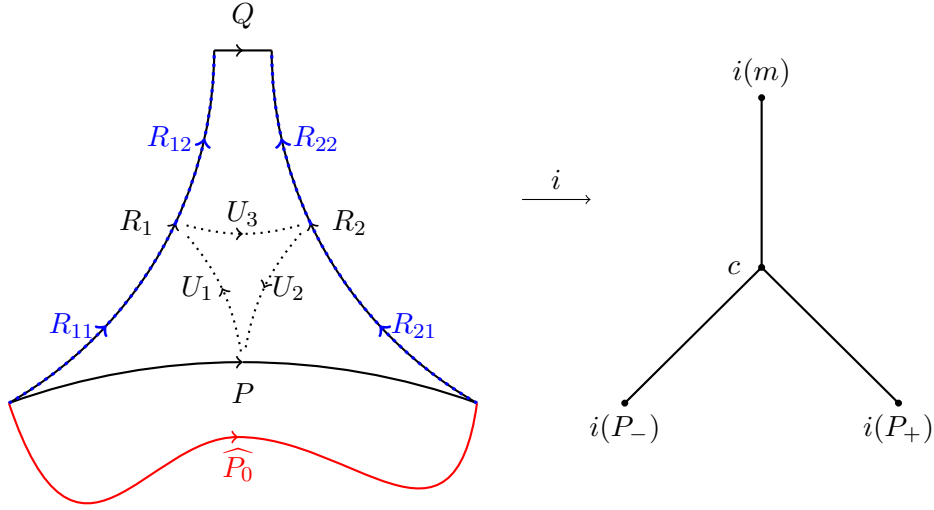
\begin{figure}[ht]
\begin{tikzpicture}[scale=0.9]
\draw[thick,->] (0,0) node (p-) {} arc (110:90:10) node[label=below:$P$] (p) {};
\draw[thick] (p) arc (90:70:10) node (p+) {};
\draw[thick,->] (p-) arc (-60:-25:6) node[label=left:$R_1$] (r1) {};
\draw[thick] (r1) arc (-25:0:6) node (q-) {};
\draw[thick,->] (p+) arc (240:205:6) node[label=right:$R_2$] (r2) {};
\draw[thick] (r2) arc (205:180:6) node (q+) {};
\draw[thick,->-] (q-.center) -- (q+.center) node[midway,label=above:$Q$] (q) {};
\draw[thick,red,->-] (p-.center) .. controls (1,-3) and (2,-0.5) .. (3.35,-0.5)  node[below] {$\widehat{P_0}$} .. controls (5,-0.5) and (6.5,-2.5) .. (p+.center);
\draw[->] (7.5,3) -- (8.5,3) node[midway,above] {$i$};
\draw[thick,fill=black] (9,0) circle (1pt) node[below] {$i(P_-)$} -- (11,2) circle (1pt) node[label=left:$c$] (c) {};
\draw[thick,fill=black] (13,0) circle (1pt) node[below] {$i(P_+)$} -- (c.center);
\draw[thick,fill=black] (11,4.5) circle (1pt) node[above] {$i(m)$} -- (c.center);
\draw[dotted,thick,->-] (p) to[bend right=15] node[midway,label={[label distance=-5pt]left:$U_1$}] (u1) {} (r1);
\draw[dotted,thick,->-] (r2) to[bend right=15] node[midway,label={[label distance=-5pt]right:$U_2$}] (u2) {} (p);
\draw[dotted,thick,->-] (r1) to[bend right=15] node[midway,label={[label distance=-5pt]above:$U_3$}] (u3) {} (r2);
\draw[blue,very thick,->-,dotted] (p-) arc (-60:-25:6) node[midway,left] {$R_{11}$};
\draw[blue,very thick,->-,dotted] (r1) arc (-25:0:6) node[midway,left] {$R_{12}$};
\draw[blue,very thick,->-,dotted] (p+) arc (240:205:6) node[midway,right] {$R_{21}$};
\draw[blue,very thick,->-,dotted] (r2) arc (205:180:6) node[midway,right] {$R_{22}$};
\end{tikzpicture}
\vspace{-4em}
\caption{The conjugacy diagram $R_1 Q R_2^{-1} P^{-1}$ (left) and the map $i$ to the tripod (right).}  \label{f:mapto3}
\end{figure}

Now we temporarily relax the assumption that endpoints of paths in Cayley graphs are always assumed to be vertices, and view all graphs as geodesic metric spaces. Let $m$ be the midpoint of the edge $Q$ and consider the isometry $i$ from the diagram $R_1 Q R_2^{-1} P^{-1}$ to a tripod $T$, such that $P_-$, $P_+$ and $m$ are mapped to the ``leaves'' of $T$. Let $c$ be the ``branching vertex'' of $T$, and let $c_1$ (resp.\ $c_2$, $c_3$) be the (unique) point in $i^{-1}(c) \cap R_2$ (resp.\ $i^{-1}(c) \cap R_1$, $i^{-1}(c) \cap P$). For each $j \in \Z/3\Z$, let $U_j$ be a geodesic in $\Cay(G,X \cup \HH)$ with $(U_j)_- = c_{j-1}$ and $(U_j)_+ = c_{j+1}$. Finally, for $j \in \Z/2\Z$, let $R_{j1}$ (resp.\ $R_{j2}$) be the subpath of $R_j$ with endpoints $(R_{j1})_- = (R_j)_-$ and $(R_{j1})_+ = c_{j+1}$ (resp.\ $(R_{j2})_- = c_{j+1}$ and $(R_{j2})_+ = (R_j)_+$). See Figure \ref{f:mapto3}.

We call a geodesic $n$-gon in a graph \emph{nice} if its vertices (as an $n$-gon) are also vertices of the graph. Now by Theorem \ref{t:hyp}, nice geodesic triangles in $\Cay(G, X \cup \HH)$ are $\nu$-slim (meaning any edge of the triangle is in the $\nu$-neighbourhood of the union of other two edges). Since any general geodesic triangle in a graph is a nice geodesic $n$-gon for $n \leq 6$, it can be shown (by drawing diagonals) that any geodesic triangle in $\Cay(G, X \cup \HH)$ is $(3\nu)$-slim. In particular, if we apply \cite[Proposition 2.1]{msri} to the conjugacy diagram $R_1 Q R_2^{-1} P^{-1}$ (viewed as a geodesic triangle with vertices $P_-$, $P_+$ and $m$), the following is true:

\begin{lem} \label{l:diams} \begin{enumerate}[(i)]
\item \label{i:ld1} The diameter of $i^{-1}(t)$ is $\leq 18\nu$ for any $t \in T$.
\item \label{i:ld2} The diameter of $i^{-1}(c)$ is $\leq 12 \nu$, i.e.\ $\ell_{X \cup \HH}(U_j) \leq 12 \nu$ for $j \in \{1,2,3\}$. \qed
\end{enumerate}
\end{lem}

We now divide the argument in two parts, depending on whether or not $Q$ is connected to a non-trivial component of $U_3$.

\begin{lem} \label{l:hlong}
There exists a universal constant $f_0 = f_0(G,Y,\{H_\omega\}_{\omega \in \Omega}) \geq 0$ such that if $Q$ is connected to a component of $U_3$ for some triple $(\hat{h},h,g)$ as above, then $$|\hat{h}|_X \geq 2|g|_X + |h|_X - 4\tilde\lambda\log_2(|\hat{h}|_X)-f_0.$$
\end{lem}

\begin{proof}
Since $R_{12}$, $R_{22}$ and $U_3$ are geodesics, and $Q$ is connected to a component of $U_3$, it follows by Lemma \ref{l:diams} \eqref{i:ld2} that $$(\ell_{X \cup \HH}(R_{12})-1) + (\ell_{X \cup \HH}(R_{22})-1) \leq \ell_{X \cup \HH}(U_3) \leq 12\nu$$ and so $$\ell_{X \cup \HH}(R_{12}) = \ell_{X \cup \HH}(R_{22}) \leq 6\nu+1.$$ By Lemma \ref{l:diams}, it follows that given any vertex $r$ on $R_{12}$ (resp.\ $R_{22}$), there exists a vertex $p$ on $P$ such that we have $|p^{-1}r|_{X \cup \HH} \leq 18\nu+1$ and $|P_-^{-1}p|_{X \cup \HH} \leq |P_-^{-1}r|_{X \cup \HH}$ (resp.\ $|P_+^{-1}p|_{X \cup \HH} \leq |P_+^{-1}r|_{X \cup \HH}$). It is easy to check that in this case $R_1QR_2^{-1}$ is a $(1,36\nu+3)$-quasi-geodesic.

Note also that the path $R_1QR_2^{-1}$ does not backtrack: indeed, if $Q$ was connected to a non-trivial component of either $R_1$ or $R_2$ then it would be connected to both of them, and if some non-trivial components of $R_1$ and $R_2$ were connected then, since $R_1$ and $R_2$ are geodesics both labelled by $g$, this would contradict the minimality of $|g|_{X \cup \HH}$.

Now consider the (phase) vertices $Q_-$ and $Q_+$ of $R_1QR_2^{-1}$. Applying Theorem \ref{t:bcp} \eqref{i:t-bcp-near} to $R_1QR_2^{-1}$ and $P$ and Corollary \ref{c:logclose} to $P_0$ and $P$, it follows that there are vertices $p_-$ and $p_+$ on $P_0$ such that $$|Q_-^{-1}p_-|_X, |Q_+^{-1}p_+|_X \leq \tilde\lambda\log_2(|\hat{h}|_X)+\tilde{c}+\varepsilon(1,36\nu+3).$$ It follows that there exist elements $$z_-,z_+ \in B_{G,X}(\tilde\lambda\log_2(|\hat{h}|_X)+\tilde{c}+\varepsilon(1,36\nu+3))$$ such that $$g = p_1 z_- = p_3^{-1}z_+ \qquad \text{and} \qquad h = z_-^{-1}p_2z_+$$ where $p_1,p_2,p_3 \in G$ are such that $$\hat{h} = p_1p_2p_3 \qquad \text{and} \qquad |\hat{h}|_X = |p_1|_X + |p_2|_X + |p_3|_X.$$ Thus setting $f_0 := 4(\tilde{c}+\varepsilon(1,36\nu+3))$ gives the result.
\end{proof}

Now consider the paths $R_{12}Q$ and $U_3R_{22}$ with endpoints $(U_3)_-$ and $Q_+$. Since $Q$ is a single edge, it follows from Lemma \ref{l:diams} \eqref{i:ld2} that both of these paths are $(1,24\nu)$-quasi-geodesics. As follows from the proof of Lemma \ref{l:hlong}, $R_{12}Q$ does not backtrack; $U_3R_{22}$ might backtrack, but we can ``shorten this path along any backtracks'' to find a  $(1,24\nu)$-quasi-geodesic path $\widetilde{U}$ with $\widetilde{U}_- = (U_3)_-$ and $\widetilde{U}_+ = Q_+$ such that all the vertices of $\widetilde{U}$ are on $U_3R_{22}$. Applying Theorem \ref{t:bcp} \eqref{i:t-bcp-conn} to $R_{12}Q$ and $\widetilde{U}$ then says that if $|h|_X > \varepsilon(1,24\nu)$ then $Q$ is connected to a non-trivial component of $\widetilde{U}$. As the path $QR_{22}^{-1}$ does not backtrack, in this case $Q$ cannot be connected to a component of $R_{22}$ (apart from $(R_{22})_+ = Q_+$ if it is a trivial component of $R_{22}$), and so $Q$ must be connected to a component of $U_3$.

Therefore Lemma \ref{l:hlong} applies whenever $|h|_X > \varepsilon(1,24\nu)$, thus for all but finitely many elements $h$. By setting $D := D_0 B_{G,X}(\varepsilon(1,24\nu))+1$ (for some $D_0 \geq 0$) and picking $f: \Z_{\geq 0} \to \Z_{\geq 0}$ to be the pointwise maximum of finitely many sublinear functions (so still sublinear), Theorem \ref{t:parbound1} follows from the following result (since $1 \in H_\omega$):

\begin{lem} \label{l:hshort}
Let $h_0 \in H_\omega$ for some $\omega \in \Omega$, and let $\mathcal{P}(h_0)$ be the set of elements $\hat{h}$ in the conjugacy class $h_0^G$ such that if $g$ and $h$ are as above, then $h = h_0$ and $Q$ is not connected to a component of $U_3$. Then $$|\mathcal{P}(h_0) \cap B_{G,X}(n)| \leq D_0 \left| B_{G,X}\left( \left\lfloor \frac{n+f_0}{2} \right\rfloor \right) \right|$$ for some constants $D_0,f_0 \geq 0$.
\end{lem}

\begin{proof}
Consider the closed path $R_{12}QR_{22}^{-1}U_3^{-1}$. Since the path $R_{12}QR_{22}^{-1}$ does not backtrack and by assumption $Q$ is not connected to a conponent of $U_3$, it follows that given any two distinct non-trivial connected components of $R_{12}QR_{22}^{-1}U_3^{-1}$, one of them must be on $U_3$ and the other one on either $R_{12}$ or $R_{22}$. But since $U_3$ is an arbitrary geodesic, we may without loss of generality assume that (loosely speaking) $U_3$ follows $R_{12}$ (and $U_3^{-1}$ follows $R_{22}$) until the last of their non-trivial connected components. Thus we have a closed path $\mathcal{C} := \widetilde{R_{12}}Q\widetilde{R_{22}}^{-1}\widetilde{U_3}^{-1}$ which does not backtrack and all of its vertices are phase except for, possibly, endpoints of $\widetilde{U_3}$. See Figure \ref{f:cycle}. Note that $\ell_{X \cup \HH}(\widetilde{U_3}) \geq 1$ since (by minimality of $|g|_{X \cup \HH}$) $R_1 \cap R_2 = \varnothing$.

\begin{figure}[ht]
\begin{tikzpicture}
\draw[thick,->] (0,0) node (r1) {} arc (-25:-15:10) node[label=left:$R_{12}$] (r12) {};
\draw[thick] (r12) arc (-15:0:10) node (q-) {};
\draw[very thick,blue,dotted,->-] (r12) arc (-15:0:10) node[midway,label=left:$\widetilde{R_{12}}$] (r12t) {};
\draw[thick,->] (3.5,0) node (r2) {} arc (205:190:10) node[label=right:$R_{22}$] (r22) {};
\draw[thick] (r22) arc (190:180:10) node (q+) {};
\draw[very thick,blue,dotted,->-] (r22) arc (190:180:10) node[midway,label=right:$\widetilde{R_{22}}$] (r22t) {};
\draw[thick] (q-.center) -- (q+.center);
\draw[very thick,blue,dotted,->-] (q-.center) -- (q+.center) node[midway,label=above:$Q$] (q) {};
\draw[very thick,blue,dotted,->-] (r12.center) to [bend right=20] node[midway,label=above:$\widetilde{U_3}$] (u3) {} (r22.center);
\draw[very thick,red,dotted,->-] (r1) arc (-25:-15:10) node[midway,label=right:$\widetilde{R_{12}'}$] (r12p) {};
\draw[very thick,red,dotted,->-] (r2) arc (205:190:10) node[midway,label=left:$\widetilde{R_{22}'}$] (r22p) {};
\fill (r1) node[above left] {$(U_3)_-$} circle (1.5pt);
\fill (r2) node[above right] {$(U_3)_+$} circle (1.5pt);
\end{tikzpicture}
\caption{The closed paths $R_{12}QR_{22}^{-1}U_3^{-1}$ (in red and blue) and $\mathcal{C} = \widetilde{R_{12}}Q\widetilde{R_{22}}^{-1}\widetilde{U_3}^{-1}$ (in blue). Here $U_3 = \widetilde{R_{12}'}\widetilde{U_3}\widetilde{R_{22}'}^{-1}$, where $\widetilde{R_{12}'}$ and $\widetilde{R_{22}'}$ are such that $R_{j2} = \widetilde{R_{j2}'} \widetilde{R_{j2}}$.}  \label{f:cycle}
\end{figure}
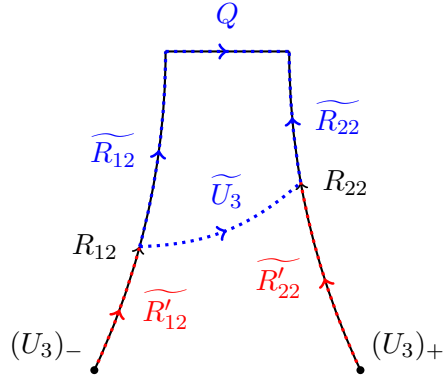

Suppose without loss of generality that one of the following three cases holds: \begin{enumerate}[(i)]
\item $\ell_{X \cup \HH}(\widetilde{R_{12}}) > \ell_{X \cup \HH}(\widetilde{R_{22}})$, or
\item $\ell_{X \cup \HH}(\widetilde{R_{12}}) = \ell_{X \cup \HH}(\widetilde{R_{22}})$ and $(\widetilde{U_3})_-$ is a phase vertex of the closed path $\mathcal{C}$, or
\item $\ell_{X \cup \HH}(\widetilde{R_{12}}) = \ell_{X \cup \HH}(\widetilde{R_{22}})$ and both $(\widetilde{U_3})_-$, $(\widetilde{U_3})_+$ are non-phase vertices of $\mathcal{C}$.
\end{enumerate} Let $r_+$ be $(\widetilde{U_3})_+$ if $(\widetilde{U_3})_+$ is a phase vertex of $\mathcal{C}$, and let $r_+$ be the vertex on $\widetilde{R_{22}}$ that is adjacent to $(\widetilde{U_3})_+ = (\widetilde{R_{22}})_-$ otherwise. It follows that both $r_+$ and $r_- := \hat{h}^{-1}r_+$ (the latter one being a vertex of $\widetilde{R_{12}}$) are phase vertices of $\mathcal{C}$. Moreover, there is a subpath of $\mathcal{C}$ with endpoints $r_-$ and $r_+$ that is a union of at most $$\ell_{X \cup \HH}(\widetilde{U_3}) + \left| \ell_{X \cup \HH}(\widetilde{R_{12}'}) - \ell_{X \cup \HH}(\widetilde{R_{22}'}) \right| + 1 \leq \ell_{X \cup \HH}(U_3)+1 \leq 12\nu+1$$ components. All of these components have length $1$ (apart from, possibly, one or two components which have length $2$), hence $|r_-^{-1}r_+|_{X \cup \HH} \leq 12\nu+3$.

Now consider the two subpaths of $\mathcal{C}$ joining $r_-$ and $Q_+$. By the above, it follows that they are $(1,24\nu+6)$-quasi-geodesics that do not backtrack and do not have non-trivial connected components. By Theorem \ref{t:bcp} \eqref{i:t-bcp-conn}, if $S$ is a component of $\mathcal{C}$, then $|S_-^{-1}S_+|_X \leq \varepsilon(1,24\nu+6)$. Thus, by the previous paragraph, it follows that $$|r_-^{-1}r_+|_X \leq (12\nu+1) \varepsilon(1,24\nu+6).$$ In particular, since $\hat{h} = r_-r_0r_-^{-1}$ where $r_0 = r_-^{-1}r_+$, by setting $$D_0 := B_{G,X}((12\nu+1) \varepsilon(1,24\nu+6))$$ and fixing $r_0 \in G$ it is enough to show that $$|\mathcal{P}(h_0,r_0) \cap B_{G,X}(n)| \leq \left| B_{G,X}\left( \left\lfloor \frac{n+f_0}{2} \right\rfloor \right) \right|$$ for some constant $f_0 \geq 0$, where $\mathcal{P}(h_0,r_0)$ is the set of $\hat{h} \in h_0^G$ with $h = h_0$ and $r_0$ as above.

Note that every vertex $v$ on $\widetilde{U_3}$ satisfies $|r_-^{-1}v|_X \leq (12\nu+1) \varepsilon(1,24\nu+6)$. Since $\widetilde{U_3}$ is an arbitrary geodesic (after fixing its endpoints), we may assume (similarly to the case above) that $\widetilde{U_3}$ follows $\widetilde{R_{12}'}^{-1} R_{11}^{-1}$ (and $\widetilde{U_3}^{-1}$ follows $\widetilde{R_{22}'}^{-1} R_{21}^{-1}$) until the last of their non-trivial connected components. Thus we obtain a path $\widetilde{R_1} \widetilde{U_3'} \widetilde{R_2}^{-1}$ (where $\widetilde{R_1}$, $\widetilde{U_3'}$, $\widetilde{R_2}$ are subpaths of $R_1$, $\widetilde{U_3}$, $R_2$, respectively) that does not backtrack and all its vertices are phase, except for possibly endpoints of $\widetilde{U_3}'$. Since all vertices of this path are on the $(1,36\nu)$-quasi-geodesic $R_{11}U_3R_{21}^{-1}$, it follows that $\widetilde{R_1} \widetilde{U_3'} \widetilde{R_2}^{-1}$ is also a $(1,36\nu)$-quasi-geodesic.

Now if either $\ell_{X \cup \HH}(\widetilde{U_3'}) > 1$ or $(\widetilde{U_3'})_-$ is a phase vertex of $\widetilde{R_1} \widetilde{U_3'} \widetilde{R_2}^{-1}$, then $\widetilde{R_1} \widetilde{U_3'} \widetilde{R_2}^{-1}$ contains a phase vertex that is on $\widetilde{U_3}$. Otherwise, consider the path $\widetilde{R_1'} \widetilde{U_3''} \widetilde{R_2}^{-1}$ obtained by replacing the non-geodesic subpath of length $2$ with interior point $(\widetilde{U_3'})_-$ by an edge $\widetilde{U_3''}$. Since $R_1$ and $R_2$ have no connected components, we have that either $(\widetilde{U_3''})_+ = (\widetilde{U_3'})_+$ is a phase vertex of $\widetilde{R_1'} \widetilde{U_3''} \widetilde{R_2}^{-1}$, or the edge $\widetilde{U_3''}$ is labelled by an element of $H_\omega \cap H_{\widetilde\omega}$ for some distinct $\omega,\widetilde\omega \in \Omega$.

Thus in either case there exists a vertex $v$ of $\widetilde{U_3}$ and some $w \in G$ such that $vw$ is a phase vertex of a $(1,36\nu)$-quasi-geodesic with endpoints $P_-$ and $P_+$ that does not backtrack, and such that $$|w|_X \leq E_0 := \max \{ |k|_X \mid k \in H_\omega \cap H_{\widetilde\omega}, \omega,\widetilde\omega \in \Omega, \omega \neq \widetilde\omega \},$$ where the right hand side is defined and finite by Theorem \ref{t:almaln}. Therefore, by Theorem \ref{t:bcp} \eqref{i:t-bcp-near} and Corollary \ref{c:logclose} it follows that this vertex $v$ satisfies $|v^{-1}u|_X \leq f_1$ for some vertex $u$ of $\widehat{P_0}$, where $$f_1 := (12\nu+1)\varepsilon(1,24\nu+6) + E_0 + \varepsilon(1,36\nu) + \tilde\lambda\log_2(|\hat{h}|_X)+\tilde{c}.$$ It follows that there exists some $z \in B_{G,X}(f_1)$ such that $$r_- = p_1z = p_2^{-1}zr_0$$ where $p_1,p_2 \in G$ are such that $$\hat{h} = p_1p_2 \qquad \text{and} \qquad |\hat{h}|_X = |p_1|_X + |p_2|_X.$$ Thus setting $f_0 := 2f_1+(12\nu+1)\varepsilon(1,24\nu+6)$ implies that $|r_-|_X \leq \frac{|\hat{h}|_X+f_0}{2}$, which gives the result.
\end{proof}

\section{Exponential negligibility of $\mathcal{P}$}
\label{s:pftmain}

This section is dedicated to a proof of Theorem \ref{t:main}. We need the following definition.

\begin{defn}
For a group $K$ with a finite generating set $Z$, the \emph{(exponential) growth rate} of $K$ with respect to $X$ is the limit $$\mu(K,Z) := \lim_{n \to \infty} \sqrt[n]{|B_{K,Z}(n)|}.$$ By submultiplicativity of ball sizes in $\Cay(K,Z)$ and a well-known result called Fekete's Lemma, it follows that this limit always exists and is equal to $\inf \{ \sqrt[n]{|B_{K,Z}(n)|} \mid n \in \Z_{\geq 0} \}$.
\end{defn}

In order to prove Theorem \ref{t:main}, we use growth tightness of relatively hyperbolic groups:

\begin{thm}[\cite{yang}, Corollary 1.7] \label{t:grtight}
Suppose $G$ is not virtually cyclic. Let $X$ be a finite generating set for $G$, and let $N \trianglelefteq G$ be an infinite normal subgroup. Let $\overline{X}$ be the image of $X$ under the quotient map $G \to G/N$, so that $\overline{X}$ is a finite generating set of $G/N$. Then $\mu(G/N,\overline{X}) < \mu(G,X)$.
\end{thm}

We also use Dehn filling in relatively hyperbolic groups, namely the following result.

\begin{thm}[\cite{osin07}, Theorem 1.1 (1)] \label{t:dfill}
Let $G$ be hyperbolic relative to a collection of subgroups $\{ H_\omega \}_{\omega \in \widetilde\Omega}$. Then there exists a finite subset $\mathcal{F}$ of $G \setminus \{1\}$ with the following property. Let $\{ N_\omega \}_{\omega \in \widetilde\Omega}$ be a collection of normal subgroups $N_\omega \trianglelefteq H_\omega$ such that $N_\omega \cap \mathcal{F} = \varnothing$ for each $\omega \in \widetilde\Omega$, and define a normal subgroup $N := \left\langle\!\left\langle \bigcup_{\omega \in \widetilde\Omega} N_\omega \right\rangle\!\right\rangle^G \trianglelefteq G$. Then for each $\omega \in \widetilde\Omega$, the natural map $H_\omega / N_\omega \to G / N$ is injective.
\end{thm}

We fix a finite generating set $X$ of $G$ for the remainder of the section. The main ingredient in the proof of Theorem \ref{t:main} (apart from Theorem \ref{t:parbound1}) is the following Lemma.

\begin{lem} \label{l:expneg}
For any $\omega \in \Omega$, the subgroup $H_\omega \leq G$ is exponentially negligible in $G$ (with respect to $X$).
\end{lem}

\begin{proof}
The idea is to use results in \cite{osin06elem} and Theorem \ref{t:dfill} to find a quotient of $G$ whose growth could be compared to the growth of $H_\omega$, and then use Theorem \ref{t:grtight}.

Suppose first that there exists a normal subgroup $N \trianglelefteq G$ such that the number $M := |H_\omega \cap N|$ is finite. Then the quotient $G/N$ is generated by the set $\overline{X}$ of images of elements of $X$ under the quotient map, and clearly $$|gN|_{\overline{X}} \leq |g|_X$$ for any $g \in G$. Also, for fixed elements $g \in G$ and $t_0 \in H_\omega \cap gN$, we have $$|H_\omega \cap gN| = |\{ t_0^{-1}t \mid t \in H_\omega \cap gN \}| \leq |H_\omega \cap N| = M.$$ In particular, it follows that
\begin{align*}
|H_\omega \cap B_{G,X}(n)| &\leq \sum_{\substack{gN \in G/N \\ gN \cap B_{G,X}(n) \neq \varnothing}} |H_\omega \cap gN| \\ &\leq M | \{ gN \in G/N \mid gN \cap B_{G,X}(n) \neq \varnothing \} | \\ &\leq M |B_{G/N,\overline{X}}(n)|.
\end{align*}
Thus, by Theorem \ref{t:grtight} it follows that as long as $N$ is infinite we have $$\limsup_{n \to \infty} \sqrt[n]{\frac{|H_\omega \cap B_{G,X}(n)|}{|B_{G,X}(n)|}} < 1,$$ which implies that $H_\omega$ is exponentially negligible in $G$. Thus the problem reduces to showing that there exists an infinite normal subgroup $N \trianglelefteq G$ such that $H_\omega \cap N$ is finite.

To construct such a subgroup, we use Dehn filling in relatively hyperbolic groups. Let $g \in G$ be a hyperbolic element (i.e.\ an element of $G \setminus \mathcal{P}$) such that the order of $g$ is infinite: such an element exists by \cite[Corollary 4.5]{osin06elem}. Consider the subgroup $$E_G(g) := \{ h \in G \mid h^{-1} g^n h = g^{\pm n} \text{ for some } n \geq 1 \}.$$ Clearly $g \in E_G(g)$, and by \cite[Lemma 4.1]{osin06elem}, the index of $\langle g \rangle \cong \Z$ in $E_G(g)$ is finite. Also, by \cite[Corollary 1.7]{osin06elem}, $G$ is hyperbolic relative to the collection $\{ H_\omega \}_{\omega \in \Omega} \cup \{ E_G(g) \}$. Let $\widetilde\Omega := \Omega \sqcup \{0\}$ and let $H_0 := E_G(g)$.

Now let $\mathcal{F} \subseteq G \setminus \{1\}$ be the finite subset given by Theorem \ref{t:dfill} applied to $G$ and the collection of subgroups $\{ H_\omega \}_{\omega \in \widetilde\Omega}$. Since $\mathcal{F}$ is finite, we have $\langle g^m \rangle \cap \mathcal{F} = \varnothing$ for $m \in \Z_{\geq 1}$ large enough. Let $N_\omega := \{1\}$ for $\omega \in \Omega = \widetilde\Omega \setminus \{0\}$, and let $N_0 := \bigcap_{h \in E_G(g)} \langle h^{-1}g^mh \rangle \trianglelefteq E_G(g)$ be the \emph{normal core} of $\langle g^m \rangle$ in $E_G(g) = H_0$, i.e.\ the kernel of the action of $E_G(g)$ on the set of left cosets of $\langle g^m \rangle$ in $E_G(g)$. As the index of $\langle g^m \rangle$ in $E_G(g)$ is finite, so is the index of $N_0$ in $E_G(g)$, so in particular, as $E_G(g)$ is infinite, so is $N_0$. Therefore, applying Theorem \ref{t:dfill} yields an infinite normal subgroup $N = \langle\!\langle N_0 \rangle\!\rangle^G$ of $G$ such that, for each $\omega \in \Omega$, the group $H_\omega \cap N$ is trivial, hence finite.
\end{proof}

\begin{proof}[Proof of Theorem \ref{t:main}]
The proof follows from Theorem \ref{t:parbound1} and Lemma \ref{l:expneg}. In particular, by Lemma \ref{l:expneg} it follows that there exists a constant $\rho > 1$ such that $\frac{|H_\omega \cap B_{G,X}(n)|}{|B_{G,X}(n)|} \leq \rho^{-n}$ for all sufficiently large $n$. Note that we have $\mu := \mu(G,X) > 1$: otherwise, if $\mu = 1$, Theorem \ref{t:grtight} would imply that $N=G$ is \emph{not} an infinite normal subgroup of $G$, and so $G$ is finite, contradicting the assumption that $G$ is not virtually cyclic.

Now choose a constant $\varepsilon > 0$ such that $\varepsilon < \mu (\min\{\mu,\rho\}-1).$ Then there exists a constant $n_0 \in \Z_{\geq 0}$ such that $$\frac{|H_\omega \cap B_{G,X}(n)|}{|B_{G,X}(n)|} \leq \rho^{-n} \qquad \text{and} \qquad |S_{G,X}(n)| \leq |B_{G,X}(n)| \leq (\mu+\varepsilon)^n$$ for all $n \geq n_0$; note also that $|B_{G,X}(n)| \geq \mu^n$ for all $n \in \Z_{\geq 0}$.

Now let $n \geq 3n_0$. By Theorem \ref{t:finmanypar}, it follows that it is enough to find an exponential (with base $<1$) upper bound on the number $$\sum_{i=0}^{\left\lfloor \frac{n+f(n)}{2} \right\rfloor} \frac{|S_{G,X}(i)| |H_\omega \cap B_{G,X}(n+f(n)-2i)|}{|B_{G,X}(n)|}$$ for any fixed $\omega \in \Omega$. We do this in three parts.

For $i < n_0$, we have
\begin{align*}
\sum_{i=0}^{n_0-1} &\frac{|S_{G,X}(i)| |H_\omega \cap B_{G,X}(n+f(n)-2i)|}{|B_{G,X}(n)|} \\ &\leq \sum_{i=0}^{n_0-1} |S_{G,X}(i)| \left(\frac{\mu+\varepsilon}{\rho}\right)^{n+f(n)-2i} \mu^{-n} \\ &\leq |B_{G,X}(n_0-1)| \left(\frac{\mu+\varepsilon}{\rho}\right)^{f(n)} \left(\frac{\mu+\varepsilon}{\rho\mu}\right)^n
\end{align*}
and $\frac{\mu+\varepsilon}{\rho\mu} < 1$ by the choice of $\varepsilon$, so since $\frac{f(n)}{n} \to 0$ as $n \to \infty$ we get exponential convergence to zero, as required.

For $i \geq n_0$ and $n+f(n)-2i \geq n_0$, we have
\begin{align*}
\sum_{i=n_0}^{\left\lfloor \frac{n+f(n)-n_0}{2} \right\rfloor} &\frac{|S_{G,X}(i)| |H_\omega \cap B_{G,X}(n+f(n)-2i)|}{|B_{G,X}(n)|} \\ &\leq \sum_{i=n_0}^{\left\lfloor \frac{n+f(n)-n_0}{2} \right\rfloor} (\mu+\varepsilon)^i \left(\frac{\mu+\varepsilon}{\rho}\right)^{n+f(n)-2i} \mu^{-n} \\ &= \sum_{i=n_0}^{\left\lfloor \frac{n+f(n)-n_0}{2} \right\rfloor} \left(\frac{\mu+\varepsilon}{\rho\mu}\right)^{n} \left( \frac{\rho^2}{\mu+\varepsilon} \right)^i \left(\frac{\mu+\varepsilon}{\rho}\right)^{f(n)}.
\end{align*}
Now if $\frac{\rho^2}{\mu+\varepsilon} \leq 1$, the result follows immediately as above. If instead $\frac{\rho^2}{\mu+\varepsilon} > 1$, we can bound the above expression by
\begin{align*}
\left(\frac{\mu+\varepsilon}{\rho\mu}\right)^{n} \frac{\left( \frac{\rho}{\sqrt{\mu+\varepsilon}} \right)^{n+f(n)+2}}{\frac{\rho^2}{\mu+\varepsilon}-1} \left(\frac{\mu+\varepsilon}{\rho}\right)^{f(n)} = \frac{(\sqrt{\mu+\varepsilon})^{f(n)}}{1-\frac{\mu+\varepsilon}{\rho^2}} \left(\frac{\sqrt{\mu+\varepsilon}}{\mu}\right)^n
\end{align*}
and since $\frac{\mu+\varepsilon}{\mu^2} < 1$ by the choice of $\varepsilon$, we get exponential convergence as before.

Finally, for $n+f(n)-2i < n_0$, we have
\begin{align*}
\sum_{i=\left\lfloor \frac{n+f(n)-n_0}{2} \right\rfloor+1}^{\left\lfloor \frac{n+f(n)}{2} \right\rfloor} &\frac{|S_{G,X}(i)| |H_\omega \cap B_{G,X}(n+f(n)-2i)|}{|B_{G,X}(n)|} \\ &\leq \sum_{i=\left\lfloor \frac{n+f(n)-n_0}{2} \right\rfloor+1}^{\left\lfloor \frac{n+f(n)}{2} \right\rfloor} (\mu+\varepsilon)^i |H_\omega \cap B_{G,X}(n+f(n)-2i)| \mu^{-n} \\ &\leq \frac{(\sqrt{\mu+\varepsilon})^{n+f(n)+2}}{\mu+\varepsilon-1} |H_\omega \cap B_{G,X}(n_0-1)| \mu^{-n} \\ &= \frac{|H_\omega \cap B_{G,X}(n_0-1)| (\sqrt{\mu+\varepsilon})^{f(n)}}{1-\frac{1}{\mu+\varepsilon}} \left(\frac{\sqrt{\mu+\varepsilon}}{\mu}\right)^n
\end{align*}
and so we again get exponential convergence. It follows that $\mathcal{P}$ is exponentially negligible.
\end{proof}

\section{Degree of commutativity}
\label{s:cor}

This section is dedicated to the proofs of Corollaries \ref{c:loxo} and \ref{c:dc}.

Given Theorem \ref{t:main}, the proof of Corollary \ref{c:loxo} is easy. Indeed, it is easy to check (either directly from Definition \ref{d:rh}, or by using characterisation of hyperbolically embedded subgroups given by \cite[Theorem 1.5]{osin06elem}) that if a group $G$ is hyperbolic relative to $\{ H_\omega \}_{\omega \in \Omega}$ and $F \leq G$ is any finite subgroup, then $G$ is hyperbolic relative to $\{ H_\omega \}_{\omega \in \Omega} \cup \{F\}$. But there are only finitely many conjugacy classes of finite-order hyperbolic elements in $G$ (i.e.\ elements of $G \setminus \mathcal{P}$) \cite[Theorem 4.2]{osin06def}, hence there exists a finite collection $\{ F_1,\ldots,F_m \}$ of finite cyclic subgroups of $G$ such that any hyperbolic element of finite order is conjugate to an element of one of the $F_j$. Thus $G$ is hyperbolic relative to $\{ H_\omega \}_{\omega \in \Omega} \cup \{ F_1,\ldots,F_m \}$, and with this structure of a relatively hyperbolic group every hyperbolic element of $G$ has infinite order. Corollary \ref{c:loxo} then follows directly from Theorem \ref{t:main}.

Given the previous paragraph, we may without loss of generality assume that all hyperbolic elements in $G$ have infinite order. The proof of Corollary \ref{c:dc} then follows closely the proof of \cite[Theorem 1.7]{amv}, stating an analogous result for ordinary hyperbolic groups. The following Lemma has been stated as Lemma 3.1 in \cite{amv} in a slightly different form, but the proof remains the same. For a group $G$ and an element $g \in G$, let $C_G(g)$ denote the centraliser of $g$ in $G$.

\begin{lem} \label{l:amv}
Let $G$ be a group generated by a finite subset $X$, and let $\mathcal{N} \subseteq G$ be a subset such that \begin{enumerate}[(i)]
\item \label{i:l-amv-negl} $\mathcal{N}$ is exponentially negligible in $G$ with respect to $X$, and
\item \label{i:l-amv-cent} there exist constants $\rho > 1$ and $n_0 \geq 0$ such that $$|C_G(g) \cap B_{G,X}(n)| \leq \rho^{-n} |B_{G,X}(n)|$$ for all $g \in G \setminus \mathcal{N}$ and $n \geq n_0$.
\end{enumerate}
Then $\{ (x,y) \in G^2 \mid xy=yx \}$ is exponentially negligible in $G$ with respect to $X$.
\end{lem}

Thus, by taking $\mathcal{N} = \mathcal{P}$, in the view of Corollary \ref{c:loxo} it is enough to show Lemma \ref{l:amv} \eqref{i:l-amv-cent}.

\begin{proof}[Proof of Corollary \ref{c:dc}]
It is known \cite[Lemma 4.1]{osin06elem} that given any hyperbolic element $g \in G$, the centraliser $C_G(g)$ contains $\langle g \rangle$ as a finite index subgroup; in particular, $C_G(g)$ is a $2$-ended subgroup of $G$. In this case, a classical result \cite[Lemma 4.1]{wall} tells that $C_G(g)$ fits into an exact sequence $$1 \to F \to C_G(g) \to Q \to 1$$ where $F$ is finite and $Q$ is either $\Z$ or $C_2 \ast C_2$. Thus $C_G(g)$ has a subgroup $K$ of index at most $2$ such that $K/F \cong \Z$ for a finite normal subgroup $F \trianglelefteq K$.

Note that $K$ contains $g^2 \in G \setminus \mathcal{P}$ and so $K$ is a $2$-ended subgroup not contained in $\mathcal{P}$. Since $F \leq K$ is a finite subgroup, \cite[Lemma 9.4]{louder} tells that there is a universal constant $m_0$ (independent of $g$) such that $|F| \leq m_0$. Since the sequence $$1 \to F \to K \to \Z \to 1$$ splits, this means that $C_G(g)$ contains a normal subgroup $\langle k_g \rangle \cong \Z$ of index $\leq 2m_0$. It is clear that $k_g \notin \mathcal{P}$: otherwise $C_G(g) \cap H_\omega^{g_0}$ is infinite for some $\omega \in \Omega$ and $g_0 \in G$, which cannot happen since $$C_G(g) \cap H_\omega^{g_0} \leq H_\omega^{g_0g} \cap H_\omega^{g_0} = \left( H_\omega^{g_0gg_0^{-1}} \cap H_\omega \right)^{g_0}$$ and $H_\omega^{g_0gg_0^{-1}} \cap H_\omega$ is finite by Theorem \ref{t:almaln} since $g \notin \mathcal{P}$.

Now consider the \emph{translation length function} in $G$, defined as $$\tau(g) := \limsup_{n \to \infty} \frac{|g^n|_{X \cup \HH}}{n} = \inf \left\{ \frac{|g^n|_{X \cup \HH}}{n} \;\middle|\; n \in \Z_{\geq 0} \right\}$$ for $g \in G \setminus \mathcal{P}$, where the second inequality follows from Fekete's Lemma since the sequence $(|g^n|_{X \cup \HH})_{n=0}^\infty$ is subadditive. It is known \cite[Theorem 4.25]{osin06def} that there exists a universal constant $\zeta > 0$ such that $\tau(g) \geq \zeta$ for all $g \in G \setminus \mathcal{P}$.

Finally, pick an element $g \in G \setminus \mathcal{P}$. Then $\tau(k_g) \geq \zeta$ for $k_g \in G \setminus \mathcal{P}$ as above, and so if $k_g^m \in B_{G,X}(n) \subseteq B_{G,X \cup \HH}(n)$, then $|m| \leq n/\zeta$. It follows that $$|\langle k_g \rangle \cap B_{G,X}(n)| \leq \frac{2n}{\zeta}+1.$$ Moreover, if $g_0 \langle k_g \rangle \cap B_{G,X}(n) \neq \varnothing$ for some $g_0 \in G$ then we may assume that $g_0 \in B_{G,X}(n)$ and so $g_0 \langle k_g \rangle \cap B_{G,X}(n) \subseteq g_0 (\langle k_g \rangle \cap B_{G,X}(2n))$. Thus $$|g_0\langle k_g \rangle \cap B_{G,X}(n)| \leq \frac{4n}{\zeta}+1.$$ Since the index of $\langle k_g \rangle$ in $C_G(g)$ is $\leq 2m_0$, this implies that $$|C_G(g) \cap B_{G,X}(n)| \leq 2m_0 \left( \frac{4n}{\zeta}+1 \right),$$ which gives a linear (and so subexponential) bound independent of $g$. Since by assumption $G$ is not virtually cyclic, the growth rate of $G$ is $\mu(G,X) > 1$, and so the proof is complete.
\end{proof}

\bibliographystyle{amsplain}
\bibliography{../../all}

\end{document}